\documentclass[12pt]{amsart}

\usepackage{amsmath,amsxtra,amsthm,amssymb,amsfonts} 

\theoremstyle{plain}

\newtheorem{thm}{Theorem}[section]

\newtheorem{defi}[thm]{Definition}
\newtheorem{prop}[thm]{Proposition}

\newtheorem{lem}[thm]{Lemma}

\newtheorem{con}[thm]{Conjecture}

\theoremstyle{definition}

\newtheorem{rem}[thm]{Remark}

\def\nat{\mathbb{N}}
\def\rls{\mathbb{R}}
\def\eucl{\mathbb{R}}    
\def\exrls{(-\infty,\infty]}

\def\lam{\lambda}
\def\wto{\stackrel{w}{\to}}
\def\ol{\overline}

\def\argmin{\operatornamewithlimits{\arg\min}}

\def\fix{\operatorname{Fix}}
\def\Min{\operatorname{Min}}
\def\cldom{\operatorname{\ol{dom}}}
\def\as{\!\mathrel{\mathop:}=}    

\def\fun{\col\hs\to\exrls}
\def\map{\col\hs\to\hs}
\def\geo{\col[0,1]\to\hs}

\def\col{\colon}
\def\l{\left}
\def\r{\right}

\def\di{\operatorname{d}\!}                  

\def\pd{P_{(D)}}                                     

\def\hs{\mathcal{H}}                             
\def\calm{\mathcal{M}}                         

\begin{document}
\title[Asymptotic behavior]{The asymptotic behavior of a class 
of nonlinear semigroups in Hadamard spaces} 
\author[M. Ba\v{c}\'ak \and S. Reich]{Miroslav Ba\v{c}\'ak \and Simeon Reich}
\date{\today}
\thanks{The research leading to these results has received funding from 
the European Research Council under the European Union's Seventh Framework 
Programme (FP7/2007-2013) / ERC grant agreement no. 267087.
The second author was partially supported by the Israel Science 
Foundation (Grant 389/12), the Fund for the Promotion of Research at the 
Technion and by the Technion General Research Fund.}
\dedicatory{Dedicated to Professor Andrzej Granas with appreciation and 
respect}
\address{Miroslav Ba\v{c}\'ak, Max Planck Institute, Inselstr. 22, 04 103 Leipzig, Germany}
\email{bacak@mis.mpg.de}
\address{Simeon Reich, Department of Mathematics, The Technion -- Israel 
Institute of Technology, 32000 Haifa, Israel}
\email{sreich@tx.technion.ac.il}

\begin{abstract}
We study a nonlinear semigroup associated with a nonexpansive mapping on 
an Hadamard space and establish its weak convergence to a fixed point. 
A discrete-time counterpart of such a semigroup, the proximal point algorithm,
turns out to have the same asymptotic behavior. This complements several 
results in the literature -- both classical and more recent ones. 
As an application, we obtain a new approach to heat flows in singular spaces 
for discrete as well as continuous times.
\end{abstract}

\maketitle



\section{Introduction and Main Results}
Throughout this paper, the symbol $(\hs,d)$ stands for an Hadamard space, 
that is, a~complete geodesic metric space of nonpositive curvature. 
Given a nonexpansive mapping $F\map,$ we study the asymptotic behavior of 
its resolvent and of the nonlinear semigroup it generates.

As a motivation, we first recall some known results concerning gradient 
flow theory in Hadamard spaces. Let $f\fun$ be a 
convex lower semicontinuous (lsc) function. Given $\lam>0,$ 
we define the 
\emph{resolvent} of $f$ by
\begin{equation} \label{eq:defres}
J_\lam x \; \as \argmin_{y\in \hs} \l[f(y)+\frac1{2\lam}d(x,y)^2\r],\qquad 
x\in \hs,
\end{equation}
and put $J_0 x \; \as x$ for each $x\in \hs.$ The \emph{gradient flow 
semigroup} 
corresponding to $f$ is defined by
\begin{equation}
S_t x  \as \lim_{n\to\infty} \l(J_{\frac{t}n}\r)^{n} x,\qquad x\in \cldom f, \label{eq:defsem} 
\end{equation}
for every $t\in[0,\infty).$ Gradient flow semigroups in 
Hadamard spaces have been studied by several authors 
\cite{jost-ch,mayer,stojkovic,ppa,semigroup,lie} and the 
theory can be extended to more general metric spaces~\cite{ambrosio}.

If $C\subset\hs$ is a convex set, we denote the corresponding metric 
projection by $P_C.$ 
The set of minimizers of a function $f\fun$ is denoted by $\Min f.$

\begin{thm}\cite[Theorem 3.1.1]{jost2}\label{thm:jost}
Let $f\fun$ be a~convex lsc function and $x_0\in \hs.$ Assume there exists a sequence $\l(\lam_n\r)\subset(0,\infty)$ with $\lam_n\to\infty$ such that $\l(J_{\lam_n}x_0\r)$ is a bounded sequence. Then $f$ attains its 
minimum and 
\begin{equation*}
 \lim_{\lam\to\infty} J_\lam x_0=P_{\Min f} \l(x_0\r).
\end{equation*}
\end{thm}

Recall that the \emph{proximal point algorithm} (PPA, for short) starting 
at a point $x_0\in\hs$ generates the sequence 
\begin{equation} \label{eq:ppa}
 x_n\as J_{\lam_n} x_{n-1}, \qquad n\in\nat, 
\end{equation}
where $\lam_n>0$ for each $n\in\nat.$
In contrast to Theorem~\ref{thm:jost}, it is known that the PPA 
converges only weakly.
\begin{thm}\cite[Theorem~1.4]{ppa} \label{thm:ppa}
Let $f\fun$ be a~convex lsc function attaining its minimum on $\hs.$ 
Then for an arbitrary starting point $x_0\in \hs$ and any sequence of 
positive reals $\l(\lam_n\r)$ such that $\sum_1^\infty\lam_n=\infty,$ 
the sequence $(x_n)\subset\hs$ defined by \eqref{eq:ppa} converges weakly 
to a~minimizer of $f.$
\end{thm}
It is not surprising that the gradient flow behaves in the same way.
\begin{thm}\cite[Theorem~1.5]{ppa} \label{thm:flow}
Let $f\fun$ be a~convex lsc function attaining its minimum on $\hs.$ Then, given a starting point $x_0\in\cldom f,$ the gradient flow $x_t\as S_t x_0$ 
converges weakly to a minimizer of~$f$ as $t\to\infty.$
\end{thm}
In a Hilbert space $H,$ one can define the resolvent and the semigroup for an 
arbitrary maximally monotone operator $A\col H\to2^H.$ The situation 
described above then corresponds to the case $A\as\partial f$ for a convex lsc function $f\col H\to\exrls.$ In particular, the semigroup in \eqref{eq:defsem} provides us with a solution to the parabolic problem
\begin{align*}
 \dot{u}(t) & \in -\partial f\l(u(t)\r) ,\quad t\in(0,\infty), \\
 u(0) & =u_0\in H
\end{align*}
for a curve $u:[0,\infty)\to H.$ Indeed, in this case $u(t)\as S_t u_0.$ 

In the present paper, we prove analogs of the above gradient flow results which in a Hilbert 
space~$H$ correspond to another important instance of a maximally
monotone operator, namely $A\as I-F,$ where $F\col H\to H$ is nonexpansive (that is, 
$1$-Lipschitz) and $I\col H\to H$ is the identity operator.

Let $F\map$ be a nonexpansive mapping. We now define its resolvent and 
the semigroup it generates as in~\cite{stojkovic}. Given a point $x\in 
\hs$ and a number $\lam>0,$ the mapping 
$G_{x,\lambda} : \hs \to \hs$ defined by 
\begin{equation}
G_{x,\lam}(y) \; \as \frac1{1+\lam}x+\frac{\lam}{1+\lam}Fy,\quad 
y\in \hs,
\end{equation}
is a strict contraction with Lipschitz constant $\frac{\lam}{1+\lam},$ and 
hence has a unique fixed point, which will be denoted $R_\lam x.$ 
The mapping $x \mapsto R_\lam x$ is called the \emph{resolvent} of $F.$

It is known that the limit
\begin{equation} \label{eq:exsemf}
 T_t x\as\lim_{n\to\infty} \l(R_{\frac{t}{n}}\r)^{n} x,\qquad x\in\hs,
\end{equation}
exists uniformly with respect to $t$ on each bounded subinterval of 
$[0,\infty).$ Moreover, the family $\l(T_t\r)$ is a strongly continuous semigroup of nonexpansive mappings \cite{stojkovic}. This definition appeared in \cite[Theorem 8.1]{reich-shafrir} in a similar context, namely, for a coaccretive operator on a hyperbolic space.

The following result is a counterpart of Theorem~\ref{thm:jost}. 
It was proved for the Hilbert ball in \cite[Theorem 24.1]{goebel-reich} and for a bounded Hadamard space in \cite[Theorem 26]{kirk}. 
The latter proof also works, however, without the boundedness assumption, as we demonstrate 
in Section \ref{sec:kirkproof} for the reader's convenience.
\begin{thm} \label{thm:kirk}
Let $F\map$ be a nonexpansive mapping and $x\in\hs.$ If there exists 
a sequence $\l(\lam_n\r)\subset(0,\infty)$ such that $\lam_n\to\infty$ and 
the sequence $\l( R_{\lam_n}x\r)_{n\in\nat}$ is bounded, 
then $\fix F$ is nonempty and
\begin{equation} \label{eq:kirklimit}
\lim_{\lam\to\infty} R_\lam x = P_{\fix F}(x).
\end{equation}
Conversely, if $\fix F\neq\emptyset,$ then the curve 
$\l(R_\lam x \r)_{\lam\in(0,\infty)}$ is bounded.
\end{thm}

Our results are presented in Proposition~\ref{prop:ppaf} and Theorem 
\ref{thm:semigroupf} below. In Proposition~\ref{prop:ppaf}, 
we give an algorithm which finds a fixed point of~$F.$ 
It is a counterpart of Theorem~\ref{thm:ppa}. For a general form of this 
algorithm in Hilbert spaces, see \cite[Theorem~23.41]{bauschke-combettes}. 
See also \cite[Theorem 2.6]{bruck-reich}, \cite[Corollary 7.10]{reich-shafrir}
and \cite[Theorem 4.7]{shafrir93}. The best result in Hadamard spaces 
works only with $\lam_n=\lam>0;$ see \cite[Theorem 6.4]{seville}.
\begin{prop}\label{prop:ppaf}
Let $F\map$ be a~nonexpansive mapping with at least one fixed point and let $\l(\lam_n\r)\subset(0,\infty)$ be a sequence satisfying $\sum_n\lam_n^2=\infty.$ Given a point $x_0\in \hs,$ put
\begin{equation} \label{eq:ppaf}
 x_n\as R_{\lam_n} x_{n-1} ,\qquad n\in\nat.
\end{equation}
Then the sequence $\l(x_n\r)$ converges weakly to a fixed point of $F.$
\end{prop}
Note that the assumption $\sum_n\lam_n^2=\infty$ also appears in Hilbert spaces; see 
\cite[Theorem~23.41]{bauschke-combettes}.

We also study the asymptotic behavior of the nonlinear semigroup defined 
in \eqref{eq:exsemf}. A~Hilbert ball version of this result 
appears in~\cite{reich91}.
\begin{thm} \label{thm:semigroupf}
 Let $F\map$ be a nonexpansive mapping with at least one fixed point 
and let $x_0\in\hs.$ Then $T_t x_0$ converges weakly to a fixed point 
of $F$ as $t\to\infty.$
\end{thm}

As noted in \cite[Remark, page 7]{bbr}, there exists a counterexample in 
Hilbert space showing that the convergence in Theorem \ref{thm:semigroupf} is not strong in general. This counterexample is based on an earlier work of J.-B. Baillon~\cite{baillon}.

In Section~\ref{sec:heat}, we apply Proposition~\ref{prop:ppaf} and Theorem~\ref{thm:semigroupf} to harmonic mapping theory in singular spaces and obtain the convergence of a heat flow to a solution to a Dirichlet problem under very mild assumptions. To this end, we construct discrete and continuous heat flows by~\eqref{eq:ppaf} and~\eqref{eq:exsemf}, respectively, with $F$ being the nonlinear Markov operator. To the best of our knowledge, these constructions are new and complement the existing approaches, for instance, the gradient flow of the energy functional.
 


\section{Preliminaries} \label{sec:pre}

In this section we recall several basic definitions and facts regarding 
Hadamard spaces. 
More information can be found in the books \cite{mybook,bh,jost2}.

Throughout the paper, the space $(\hs,d)$ is Hadamard, that is, it is a 
complete geodesic metric space satisfying
\begin{equation} \label{eq:cat}
d\l(x,\gamma_t\r)^2\leq 
(1-t)d\l(x,\gamma_0\r)^2+td\l(x,\gamma_1\r)^2
-t(1-t)d\l(\gamma_0,\gamma_1\r)^2
\end{equation}
for any $x\in \hs,$ any geodesic $\gamma\geo,$ and any $t\in[0,1].$ Given a 
closed and convex set $C\subset\hs$ and a point $x\in\hs,$ there exists a 
unique point $c\in C$ such that
\begin{equation*}
d(x,c) = d(x, C) := \inf_{y\in C} d(x,y).
\end{equation*}
We denote this point $c$ by $P_Cx$ and call the mapping 
$P_C\col\hs\to C$ the \emph{metric projection} of $\hs$ onto the set $C.$

Given a bounded sequence $(x_n)\subset\hs,$ put
\begin{equation}\label{eq:omega}
\omega\l(x;\l(x_n\r)\r)\as\limsup_{n\to\infty} d\l(x,x_n\r)^2,\qquad x\in \hs.
\end{equation}
Then the function $\omega$ defined in~\eqref{eq:omega} has a unique minimizer, which we call 
the \emph{asymptotic center} of the sequence $(x_n).$ We shall say that $\l(x_n\r)\subset 
\hs$ 
\emph{weakly converges} to a point $x\in\hs$ if $x$ is the asymptotic center of each subsequence of $\l(x_n\r).$ We use the notation $x_n\wto x.$ Clearly, if $x_n\to x,$ then $x_n\wto x.$ If there is a subsequence $\l(x_{n_k}\r)$ of $\l(x_n\r)$ such that $x_{n_k}\wto z$ for some $z\in \hs,$ we say that $z$ is a \emph{weak cluster point} of the sequence $\l(x_n\r).$ 

We say that a sequence $\l(x_n\r)\subset\hs$ is \emph{Fej\'er monotone} with respect to a set $A\subset\hs$ if
\begin{equation*}
d\l(a,x_{n+1}\r)\le  d\l(a,x_n\r)
\end{equation*}
for each $a\in A$ and $n\in\nat.$

\begin{prop}\cite[Proposition 3.3]{apm}
Let $C\subset \hs$ be a closed convex set. Assume that $(x_n)\subset \hs$ 
is a~Fej\'er monotone sequence with respect to $C.$ Then we have:
\begin{enumerate}
\item $(x_n)$ is bounded. \label{item:i}
\item $d(x_{n+1},C) \leq d(x_n,C)$ for each $n\in\nat.$ \label{item:ii}
\item $(x_n)$ weakly converges to some $x\in C$ if and only if all weak cluster points of $(x_n)$ belong to $C.$ \label{item:iii}
\item $(x_n)$ converges to some $x\in C$ if and only if $d(x_n,C)\to 0.$ \label{item:iv}
\end{enumerate}
\label{prop:fejer}
\end{prop}

For each $\lam>0$ and $x\in \hs,$ we have $R_\lam x=x$ if and only if $Fx=x.$ Furthermore, we have the following estimate.
\begin{lem}\cite[Lemma 3.4]{stojkovic} 
 Let $F\map$ be a nonexpansive mapping. Then its resolvent satisfies
\begin{equation} \label{eq:resestim}
d\l(x,R_\lam x\r)\leq \lam d(x,Fx)
\end{equation}
for every $\lam\in(0,\infty).$
\end{lem}
\begin{proof} Since $R_\lam x$ is a fixed point of the strict contraction $G_{x,\lam},$ 
it can be iteratively approximated by the Banach contraction principle. 
Therefore
 \begin{align*}
  d\l(x,R_\lam x\r) &\leq \sum_{n=1}^\infty 
  d\l(G_{x,\lam}^{n-1}(x),G_{x,\lam}^{n}(x) \r) \\ &\leq 
  d\l(x,G_{x,\lam}(x) \r) \sum_{n=1}^\infty \l(\frac{\lam}{1+\lam} \r)^{n-1} 
  \\ & \leq (1+\lam)d\l(x,G_{x,\lam}(x) \r)
 \end{align*}
and we are done because the right-hand side is equal to $\lam d(x,Fx).$
\end{proof}
Consequently,
\begin{equation*}
d\l(x,R_{\frac{t}{n}}^{n}x\r)\leq\sum_{j=0}^{n-1} d\l(R_{\frac{t}{n}}^{j}x,
R_{\frac{t}{n}}^{j+1}x\r) \leq nd\l(x,R_{\frac{t}{n}}x\r)=td(x,Fx)
\end{equation*}
and taking the limit on the left-hand side as $n\to\infty$, we obtain
\begin{equation} \label{eq:resestim2}
d\l(x,T_tx\r)\leq t d(x,Fx).
\end{equation}



\section{Proof of Theorem~\ref{thm:kirk}}  \label{sec:kirkproof}

\begin{proof}[Proof of Theorem~\ref{thm:kirk}]
To simplify our notation, put $x_\lam\as R_\lam x$ for each $\lam\in(0,\infty).$ Fix now $0<\mu<\lam$ and let \begin{equation*}\ol{\triangle}\l(\ol{x},\ol{Fx_\lam},\ol{Fx_\mu}\r)\subset\eucl^2\end{equation*}
be a comparison triangle for $\triangle\l(x,Fx_\lam,Fx_\mu\r).$ We have
\begin{equation*}
\l\|\ol{Fx_\lam}-\ol{Fx_\mu}\r\|  = d\l(Fx_\lam,Fx_\mu\r)  \leq d\l(x_\lam,x_\mu\r)  \leq \l\| \ol{x_\lam}-\ol{x_\mu} \r\|. 
\end{equation*}
Without loss of generality we may assume that $\ol{x}=0\in\rls^2.$ From 
this 
and the fact that $\ol{x_\lam}=\frac{\lam}{1+\lam} \ol{F x_\lam}$ and $\ol{x_\mu}=\frac{\mu}{1+\mu}\ol{Fx_\mu}$ we further obtain
\begin{equation*}\l\langle \frac{1+\lam}{\lam}\ol{x_\lam}-\frac{1+\mu}{\mu}\ol{x_\mu},
\frac{1+\lam}{\lam}\ol{x_\lam}-\frac{1+\mu}{\mu}\ol{x_\mu} \r\rangle 
\leq  \l\| \ol{x_\lam}-\ol{x_\mu} \r\|^{2}.\end{equation*}
A simple computation yields
\begin{align*}
\l(\frac{1+\lam}{\lam}-\frac{1+\mu}{\mu}\r)^2\l\|\ol{x_\mu} \r\|^{2} 
& +\l(\frac{(1+\lam)^2}{\lam^2}-1\r) \l\|\ol{x_\lam}-\ol{x_\mu} \r\|^{2} \\ 
& \leq 2\l(\frac{1+\mu}{\mu}-\frac{1+\lam}{\lam} \r)\frac{1+\lam}{\lam}\l\langle \ol{x_\mu},
\ol{x_\lam}-\ol{x_\mu}\r\rangle.
\end{align*}
Consequently,
\begin{equation*} \l\langle \ol{x_\mu},\ol{x_\lam}-\ol{x_\mu}\r\rangle\geq0.\end{equation*}
Since
\begin{equation*} \l\|\ol{x_\lam} \r\|^2 =\l\|\ol{x_\mu} \r\|^2+ \l\|\ol{x_\lam}-\ol{x_\mu} \r\|^2+2\l\langle \ol{x_\mu},\ol{x_\lam}-\ol{x_\mu}\r\rangle,\end{equation*}
we have
\begin{equation*}\l\|\ol{x_\mu} \r\|\leq \l\|\ol{x_\lam} \r\|\end{equation*}
and
\begin{equation} \label{eq:adhockirk}
d\l(x_\lam,x_\mu\r)^2 \leq \l\|\ol{x_\lam}-\ol{x_\mu} \r\|^2\leq \l\|\ol{x_\lam}-\ol{x} \r\|^2-\l\|\ol{x_\mu}-\ol{x} \r\|^2 .
\end{equation}
The monotonicity of $\lam\mapsto\l\|\ol{x_\lam} \r\|$ and the boundedness 
of the sequence $\{\ol{x_{\lam_n}}\}$ yield the boundedness of the curve 
$\l(x_\lam\r)_{\lam\in(0,\infty)}$. 
Inequality~\eqref{eq:adhockirk} therefore implies that
\begin{equation*} d\l(x_\lam,x_\mu\r)^2\to0 \quad
\text{as }\lam,\mu\to\infty.\end{equation*}
Let $z\in\hs$ be the limit point of $(x_\lam).$ Using continuity, we obtain
\begin{equation*}d(z,Fz)=\lim_{\lam\to\infty} d\l(x_\lam, Fx_\lam\r)
=\lim_{\lam\to\infty} \frac1{1+\lam} d\l(x, Fx_\lam\r)=0,\end{equation*}
which means that $z\in\fix F.$ 

We will now show that $z= P_{\fix F}(x).$ Let $p\in\fix F$ be an arbitrary fixed point 
of $F$ and repeat 
the above argument with the triangle $\triangle\l(x,p,Fx_\mu\r).$ We obtain
\begin{equation*}
d(x,p)^2 = \l\|\ol{x}-\ol{p} \r\|^2 \geq \l\|\ol{x}-\ol{x_\mu} \r\|^2
+\l\|\ol{x_\mu}-\ol{p} \r\|^2 \geq d\l(x,x_\mu\r)^2+d\l(x_\mu,p\r)^2
\end{equation*}
and after taking the limit on the right-hand side as $\mu\to\infty,$ we arrive at 
\begin{equation*} d(x,p)^2\geq d(x,z)^2+d\l(z,p\r)^2,\end{equation*}
which completes the proof that $z= P_{\fix F}(x).$

Finally, it is easy to see that if $\fix F\neq\emptyset,$ then $\l(x_\lam\r)_{\lam\in(0,\infty)}$ is bounded.
\end{proof}



\section{Proof of Proposition~\ref{prop:ppaf}}  \label{sec:ppafproof}

\begin{proof}[Proof of Proposition~\ref{prop:ppaf}]
Let $x\in \hs$ be a fixed point of $F.$ Then, for each $n\in\nat,$ we have
\begin{equation*}
d\l(x_{n-1},x\r)\geq d\l(R_{\lam_n} x_{n-1},R_{\lam_n} x\r)=d\l(x_n,x\r),
\end{equation*}
which verifies the Fej\'er monotonicity of $\l(x_n\r)$ with respect to $\fix F.$ Put
\begin{equation*}
\beta_n\as\frac1{1+\lam_n}.
\end{equation*}
Inequality~\eqref{eq:cat} yields
\begin{align*}
 d\l(x,x_n\r)^2 &\leq \beta_n d\l(x,x_{n-1}\r)^2 + (1-\beta_n) d\l(x,Fx_n\r)^2 \\ & \quad -\beta_n(1-\beta_n)d\l(x_{n-1},Fx_n\r)^2 \\
& \leq\beta_n d\l(x,x_{n-1}\r)^2+(1-\beta_n) d\l(x,x_n\r)^2-\beta_n d\l(x_{n-1},x_n\r)^2,
\end{align*}
which gives
\begin{align}
d\l(x_{n-1},x_n\r)^2 &\leq d\l(x,x_{n-1}\r)^2-d\l(x,x_n\r)^2 \nonumber \\
\intertext{and hence}
\lam_n^2\frac{d\l(x_{n-1},x_n\r)^2}{\lam_n^2} & \leq d\l(x,x_{n-1}\r)^2-d\l(x,x_n\r)^2.  \label{eq:ppaestim}
\end{align}
By the triangle inequality, we have
\begin{align*}
d\l(x_n,x_{n+1}\r)+d\l(x_{n+1},Fx_{n+1}\r) &= d\l(x_n,Fx_{n+1}\r) \\ &\leq d\l(x_n,Fx_n\r)
+d\l(Fx_n,Fx_{n+1}\r) \\ & \leq d\l(x_n,Fx_n\r)+d\l(x_n,x_{n+1}\r)
\end{align*}
and therefore
\begin{equation} \label{eq:ppamono}
 \frac{d\l(x_{n},x_{n+1}\r)}{\lam_{n+1}} =d\l(x_{n+1},Fx_{n+1}\r)
\leq d\l(x_n,Fx_n\r)=\frac{d\l(x_{n-1},x_{n}\r)}{\lam_n}. 
\end{equation}
Summing up \eqref{eq:ppaestim} over $n=1,\dots,m,$ where $m\in\nat,$ 
and using \eqref{eq:ppamono}, we obtain 
\begin{equation*}
\biggl(\sum_{n=1}^m \lam_n^2\biggr) \frac{d\l(x_{m-1},x_m\r)^2}{\lam_m^2}\leq d\l(x,x_0\r)^2-d\l(x,x_m\r)^2.
\end{equation*}
Hence 
\begin{equation*}
d\l(x_m,Fx_m\r)=\frac{1}{\lam_m}d\l(x_{m-1},x_m\r)\to 0 \qquad\text{as } m\to\infty.
\end{equation*}
Assume now that $z\in\hs$ is a weak cluster point of $\l(x_n\r).$ Then
\begin{align*}
\limsup_{n\to\infty} d\l(Fz,x_n\r) & \leq \limsup_{n\to\infty} \l[d\l(Fz,Fx_n\r)+ d\l(Fx_n,x_n\r)\r], \\
& \leq \limsup_{n\to\infty} d\l(z,x_n\r)+0.
\end{align*}
By the uniqueness of the weak limit, we get $z=Fz.$ Finally, we apply 
Proposition~\ref{prop:fejer}\eqref{item:iii} to conclude that the sequence $\l(x_n\r)$ 
weakly converges to a fixed point of~$F.$ 
\end{proof}



\section{Proof of Theorem~\ref{thm:semigroupf}}  \label{sec:semigroupfproof}

\begin{proof}[Proof of Theorem~\ref{thm:semigroupf}]
We mimic the technique from \cite{reich91} and adapt it 
to our situation. Let $x\in\hs.$ First observe that
\begin{equation*}
d\l(R_\lam x, FR_\lam x\r)=\frac1\lam d\l(x,R_\lam x\r)\leq d(x,Fx)
\end{equation*}
by~\eqref{eq:resestim}. Hence we have
\begin{equation*}
 d(x,Fx) \geq  d\l(R_{\frac{t}{n}} x, FR_{\frac{t}{n}} x\r)
\geq d\l(R_{\frac{t}{n}}^{n} x, FR_{\frac{t}{n}}^{n} x\r)
\end{equation*}
and after taking the limit on the right-hand side as $n\to\infty$, we also obtain 
\begin{equation*}
 d(x,Fx)\geq d\l(T_t x, FT_tx\r).
\end{equation*}
The semigroup property implies (when we substitute $x\as T_s x$ and $t\as 
t-s$ in the above inequality) that 
\begin{equation*}
  d\l(T_s x, FT_sx\r)\geq d\l(T_t x, FT_tx\r),
\end{equation*}
whenever $s\leq t$ and therefore the limit
\begin{equation} \label{eq:limit}
 \lim_{t\to\infty} d\l(T_t x, FT_tx\r)
\end{equation}
 exists. We will now show that this limit actually equals $0.$ 
 
Let $0\leq s\leq t.$ Then inequality \eqref{eq:resestim2} yields 
\begin{align*}
d\l(T_s x, T_t x\r)  & \leq \sum_{j=0}^{n-1} d\l(T_{s+\frac{j}{n}(t-s)} x, 
T_{s+\frac{j+1}{n}(t-s)} x\r) \\ & \leq\frac{t-s}{n} \sum_{j=0}^{n-1} 
d\l(T_{s+\frac{j}{n}(t-s)} x, F T_{s+\frac{j}{n}(t-s)} x\r)
\end{align*} 
and after letting $n\to\infty,$ we obtain
\begin{equation} \label{eq:semiestim}
d\l(T_s x, T_t x\r)  \leq \int_s^t d\l(T_r x, F T_r x\r) \di r.
\end{equation} 
Next we prove that
\begin{equation} \label{eq:limitestim}
\lim_{t\to\infty} d\l(T_t x,F T_t x\r)\leq \frac1h \lim_{t\to\infty} d\l(T_{t+h} x, T_t x\r).
\end{equation} 
To this end, we repeatedly use the inequality
\begin{align*}
d\l(FR_\lam^{n} x, R_\lam^{n-k+1}x\r) & \leq\frac{1}{1+\lam} d\l(FR_\lam^{n} x, R_\lam^{n-k}x\r) \\ &\quad +\frac{\lam}{1+\lam} d\l(R_\lam^{n} x, R_\lam^{n-k+1}x\r) ,
\end{align*}
which is valid for each $1\leq k \leq n$, to obtain
\begin{align*}
d\l(FR_\lam^{n} x, R_\lam^{n}x\r) & \leq \frac1{(1+\lam)^n} d\l(FR_\lam^{n} x, x\r) \\ & \quad +\lam\sum_{j=1}^n \frac{1}{(1+\lam)^j} d\l(R_\lam^{n} x, R_\lam^{n-j+1}x\r) .
\end{align*}
Put now $\lam\as\frac{t}{n}$ and take the limits on both sides of this inequality as 
$n\to\infty.$ One arrives at
\begin{equation*}
d\l(T_t x,F T_t x\r)\leq \int_0^t e^{-r} d\l(T_t x, T_{t-r} x\r)\di r + e^{-t} d\l( FT_tx,x\r).
\end{equation*}
Applying inequality \eqref{eq:semiestim} and an elementary calculation, 
we arrive at
\begin{equation*}
e^t d\l(T_t x,F T_t x\r) \leq \int_0^t\l(e^r-1\r) d\l(T_r x,F T_r x\r) \di r
+ d\l(FT_tx,x\r)
\end{equation*}
or
\begin{equation*}
\l(e^t-1\r) d\l(T_t x,F T_t x\r) \leq \int_0^t\l(e^r-1\r) d\l(T_r x,F T_r x\r) \di r+d\l(T_tx,x\r).
\end{equation*}
Replacing $t$ by $h$ and then $x$ by $T_t x$, we get 
\begin{align*}
\l(e^h-1\r) d\l(T_{t+h} x,F T_{t+h} x\r) & \leq \int_t^{t+h}\l(e^{r-t}-1\r) d\l(T_r x,F T_r x\r) \di r \\ & \quad +d\l(T_{t+h}x,T_t x\r).
\end{align*}
By an easy calculation, we obtain
\begin{align*}
d\l(T_{t+h}x,T_t x\r) & \geq \l(e^h-1\r) \l[d\l(T_{t+h} x,F T_{t+h} x\r) -d\l(T_{t} x,F T_{t} x\r) \r] \\ & \quad +h d\l(T_{t} x,F T_{t} x\r) ,
\end{align*}
which proves \eqref{eq:limitestim}. Now \eqref{eq:limitestim} and \eqref{eq:semiestim} 
yield
\begin{align*}
\lim_{t\to\infty} d\l(T_t x,F T_t x\r) & \leq \limsup_{h\to\infty} 
\frac1h  d\l(T_{h} x,  x\r) \\ & \leq \lim_{h\to\infty}\frac1h\int_0^h 
d\l(T_r x,F T_r x\r) \di r \\ & = \lim_{t\to\infty} d\l(T_t x,F T_t x\r) 
\end{align*}
and thus
\begin{equation} \label{eq:limsup}
 \lim_{t\to\infty} d\l(T_t x,F T_t x\r) = \limsup_{h\to\infty} \frac1h  d\l(T_{h} x, x\r).
\end{equation}
Let now $y\in\hs.$ Then
\begin{equation*} 
 \limsup_{h\to\infty} \frac{d\l(T_{h} x, x\r)}{h}
\le \limsup_{h\to\infty} \frac1h \l[ d\l(T_{h} x, T_{h}y\r) + d\l(T_{h} y,y\r)
 +d\l(y,x\r) \r]\le \limsup_{h\to\infty} \frac{d\l(T_{h} y, y\r)}{h}
\end{equation*}
and since $y$ was arbitrary, the value on the left-hand side is 
independent of~$x.$ Consequently, by virtue of~\eqref{eq:limsup}, 
the limit in \eqref{eq:limit} is independent of~$x$ and is therefore 
equal to~$0$ because one may choose $x\in\fix F.$

To finish the proof, choose a sequence $t_n\to\infty$ and set $x_n\as 
T_{t_n} x_0.$ Since $T_t$ is nonexpansive, we know that the 
sequence $(x_n)$ is Fej\'er 
monotone with respect to $\fix F.$ In particular, $(x_n)$ is bounded and therefore has a 
weak cluster point $z\in\hs.$ It suffices to show that $z\in\fix F.$ We easily get
\begin{align*}
 \limsup_{n\to\infty} d\l(Fz,x_n\r)& \leq \limsup_{n\to\infty} 
d\l(Fz,Fx_n\r) + \limsup_{n\to\infty} d\l(Fx_n,x_n\r) \\ &\leq \limsup_{n\to\infty} 
d\l(z,x_n\r),
\end{align*}
which by the uniqueness of the weak limit yields $z=Fz.$ Here we used, of 
course, the fact that the limit in~\eqref{eq:limit} is $0.$

It is easy to see that $z$ is independent of the choice of the sequence $\l(t_n\r)$ and therefore $T_t x_0\wto z.$
\end{proof}



\section{Discrete and continuous heat flows in singular spaces} \label{sec:heat}

There has been considerable interest in harmonic mappings between 
singular spaces and several (nonequivalent) approaches have been 
developed in the case of an Hadamard space target. 
See, for example, 
\cite{gromovschoen,jost94,jost97,jost-ch,korevaarschoen,ohta,sturm01,sturm02,sturm05}. 
We will follow~\cite{sturm01} and consider an $L^2$-Dirichlet problem for 
mappings from a measure space equipped with a~symmetric Markov kernel 
to an Hadamard space. Under the assumption that the Markov kernel 
satisfies an $L^2$-spectral bound condition, 
it is shown in~\cite{sturm01} that a Dirichlet problem has a unique solution 
and that an associated heat flow (defined as a gradient flow of the energy 
functional) converges to this solution. The $L^2$-spectral bound condition 
is completely natural albeit rather strong, because the heat flow 
semigroup is then a~contracting mapping and converges to the unique 
solution exponentially fast~\cite{sturm01}. Since the energy functional 
is a convex continuous function (see \cite[page 342]{sturm01}), 
one can alternatively apply Theorem~\ref{thm:flow} and conclude that the 
heat flow converges \emph{weakly} to a solution to the Dirichlet problem 
provided there exists a solution. This, of course, does not require 
a~spectral bound condition.

In the present paper we use a somewhat different approach than~\cite{sturm01} to construct discrete and continuous time heat flows, namely 
formulae~\eqref{eq:ppaf} and~\eqref{eq:exsemf}. Then we employ 
Proposition~\ref{prop:ppaf} and Theorem~\ref{thm:semigroupf} to obtain the convergence of these heat flows to a~solution to the Dirichlet problem. Let us first formulate the Dirichlet problem for singular spaces. 
For the details, see the original paper~\cite{sturm01}.

Let $(M,\calm,\mu)$ be a measure space with a $\sigma$-algebra $\calm$ 
and a measure $\mu$, and assume that it is complete in the sense that all 
subsets of a $\mu$-null set belong to~$\calm.$ 
Given a set $D\in\calm,$ define 
\begin{equation*}
L^2(D)\as\l\{u\in L^2(M)\col u=0 \text{ a.e. on } M\setminus D \r\}.
\end{equation*}

Next, let $(\hs,d)$ be an Hadamard space, fix a measurable mapping 
$h\col M\to\hs$ and consider the nonlinear Lebesgue space $L^2(D,\hs,h)$ of 
measurable mappings $f\col M\to\hs$ satisfying 
\begin{equation*}
 d\l(f(\cdot),h(\cdot)\r) \in L^2(D).
\end{equation*}
The space $L^2(D,\hs,h),$ when equipped with the metric
\begin{equation*}
d_2(f,g)\as \biggl( \int_M d\l(f(x),g(x)\r)^2 \mu(\di x) \biggr)^{\frac12},
\end{equation*}
is again an Hadamard space.

Let $p\as p(x,\di y)$ be a Markov kernel, which is symmetric with respect to $\mu,$ that is, we have $p(x,\di y) \mu(\di x) = p(y,\di x) \mu(\di y)$ for every $x,y\in M.$ Then one can define the nonlinear Markov operator $P\col L^2(M,\hs,h)\to L^2(M,\hs,h)$ by
\begin{equation*}
Pf(x)\as \argmin_{z\in\hs} \int_M d\l(z,f(y)\r)^2 p(x,\di y),
\end{equation*}
where $f \in L^2(M,\hs,h).$ By \cite[Theorem 5.2]{sturm01}, we know that
\begin{equation*}
d_2\l(Pf,Pg\r)\leq d_2\l(f,g\r) ,
\end{equation*}
for every $f,g\in L^2(M,\hs,h),$ that is, the nonlinear Markov operator is nonexpansive on $L^2(M,\hs,h).$ A fixed point of $P$ is called a harmonic mapping.
\begin{rem}
If $M\subset \rls^n$ is a bounded set and $P\col L^2(M,\rls)\to L^2(M,\rls)$ is the usual 
(linear) Markov operator, then the Laplacian satisfies $\Delta= I-P,$ 
where $I\col L^2(M,\rls)\to L^2(M,\rls)$ is the identity operator, and we see that a 
function $f\col M\to\rls$ is harmonic if $\Delta f=0.$
\end{rem}
 
Since we are concerned with the Dirichlet problem, a somewhat refined notion of a nonlinear Markov operator is needed. Given $D\in\calm,$ define a new Markov kernel
 \begin{equation*}
 p_D(x,\di y)\as \chi_D(x)p(x,\di y) +\chi_{M\setminus D}\delta_x(\di y),
 \end{equation*}
for every $x,y\in M.$ Denote by $\pd$ the nonlinear Markov operator associated with the 
kernel $p_D.$ Then we have
\begin{equation*} \pd f(x)=\l\{
  \begin{array}{ll}
     Pf(x) & \text{if }x\in D \vspace{5pt} \\ 
             f(x) & \text{if }x\in M\setminus D,
  \end{array}
 \r.
\end{equation*}
and $\pd \col L^2(D,\hs,h)\to L^2(D,\hs,h)$ is also nonexpansive.
\begin{defi}[$L^2$-Dirichlet problem]
 Let $D\in\calm$ and $h\col M\to\hs$ be a measurable mapping. Is there $f\in L^2(D,\hs,h)$ such that $\pd f=f$?
\end{defi}
In~\cite[Theorem 6.4]{sturm01}, the author shows that the Dirichlet 
problem has a unique solution provided the linear operator
\begin{equation*}
 \pd f(x)\as \int_M f(y) p_D(x,\di y),\qquad f\in L^2(M),
\end{equation*}
satisfies the spectral bound condition $\lambda_k>0$ for some $k\in\nat,$ where
\begin{equation*}
 \lambda_k\as 1-\|\pd^k\|_{L^2(D)}.
\end{equation*}
Under this assumption, one also gets strong (and exponentially fast) convergence of an associated heat flow (defined as a gradient flow of the energy) to the (unique) solution to the Dirichlet 
problem~\cite{sturm01}. In contrast, we define a heat flow by formula \eqref{eq:exsemf} for $F\as\pd$ and moreover do not assume any spectral bound condition. Applying Proposition~\ref{prop:ppaf} and Theorem~\ref{thm:semigroupf} 
with $F\as\pd$, we see that if there exists a solution to the 
Dirichlet problem 
for a measurable mapping $h\col M\to\hs,$ both the proximal point algorithm
(discrete time heat flow) and the semigroup (continuous time heat flow) 
weakly converge to a mapping $f\in L^2(D,\hs,h)$ such that $\pd f=f.$
 
We finish our paper by making the following conjecture.
\begin{con}
The heat flow semigroup $\l(T_t\r)$ defined by \eqref{eq:exsemf} 
for $F\as\pd$ coincides with the heat flow constructed 
in \cite[Theorem 8.1]{sturm01}.
\end{con}

\subsection*{Acknowledgment.} We thank the referee for providing us with 
pertinent comments and helpful suggestions.



\end{document}